\documentclass{amsart}

  \usepackage[utf8]{inputenc}
  \usepackage{lmodern}

\usepackage{amssymb,amscd}

\usepackage[usenames]{xcolor}
\usepackage[backref=page,
        colorlinks,
        citecolor=red,
        pdfauthor={Solomon Vishkautsan, Francesco Veneziano},
        pdftitle={The field of iterates of a rational function}]{hyperref}

\usepackage[alphabetic, backrefs]{amsrefs}

\usepackage[normalem]{ulem} %

\theoremstyle{plain}
\newtheorem{theorem}{Theorem}[section]
\newtheorem{prop}[theorem]{Proposition}
\newtheorem{lemma}[theorem]{Lemma}
\theoremstyle{definition}

\newtheorem{exmp}[theorem]{Example}

\theoremstyle{remark}
\newtheorem{rem}[theorem]{Remark}

\DeclareMathOperator{\lcm}{lcm}
\DeclareMathOperator{\pgl}{PGL}
\DeclareMathOperator{\gal}{Gal}

\DeclareMathOperator{\GL}{\mathrm{GL}}
\DeclareMathOperator{\PGL}{\mathrm{PGL}}
\newcommand{\Q}{\mathbb{Q}}
\newcommand{\Qbar}{\overline{\mathbb{Q}}}

\renewcommand{\P}{\mathbb{P}}

\title{The field of iterates of a rational function}
\author{Francesco Veneziano}
\address[Francesco Veneziano]{Department of mathematics, University of Genova, Via Dodecaneso~35, 16146 Ge\-no\-va, Italy}
\email{veneziano@dima.unige.it}

\author{Solomon Vishkautsan}
\address[Solomon Vishkautsan]{Department of Computer Science, Tel-Hai Academic College, Upper Galilee 9977, Qiryat Shemona 1220800, Israel}
\email{wishcow@gmail.com}


\begin{document}

\begin{abstract}
We study how the field of definition of a rational function changes under iteration. We provide a complete classification of polynomials with the property that the field of definition of one of their iterates drops in degree (over a given base field). 
We show with families of examples that this characterization does not hold for rational functions. 
Finally, we also classify fractional linear transformations with this property. 
\end{abstract}

\maketitle

\section{Introduction}
Let $K$ be a number field, and $\bar{K}$ its algebraic closure.
A rational function $f(x)$ defined over $\bar{K}$ is a quotient $p(x)/q(x)$ of two polynomials $p(x),q(x)$ with $q\neq 0$ in $\bar{K}[x]$; the set of rational functions over $\bar{K}$ is denoted by $\bar{K}(x)$. If $f(x)=p(x)/q(x)$ is a non-zero rational function and $p(x),q(x)$ are coprime polynomials, the degree of $f$ is defined as $\deg(f(x))=\max(\deg(p(x)),\deg(q(x))$. A non-zero rational function can be represented in a unique way as a ratio of coprime polynomials $p(x)/q(x)$ with $q(x)$ monic.

Let $f(x)=\frac{p(x)}{q(x)}\in{\bar{K}(x)^*}$, with $p(x)=\sum_{i=0}^{n}a_i x^i, q(x)=\sum_{i=0}^{m}b_i x^i, b_m=1$ and $p(x),q(x)$ coprime; we denote by $K(f)$ the field $K(a_0,\dotsc, a_n,b_0,\dotsc, b_m)$ defined by adjoining to $K$ all the coefficients of $p(x)$ and $q(x)$. This is the \emph{field of definition} of $f$ (over $K$); it can also be characterized as the fixed field of the subgroup of all elements in $\gal(\bar{K}/K)$ which act trivially on $f$. 

We denote by $f^{\circ{n}}$ the $n$-th iterate of $f$ under composition; i.e., $f^{\circ{1}}=f$ and  $f^{\circ{n}}=f^{\circ{(n-1)}}\circ{f}$ for any integer $n\ge{2}$. When $\deg f\geq 1$ we can consider $f$ as a dynamical system on $\P_1(\bar{K})$, which we identify with $\bar{K}\cup\{\infty\}$, and we can study the orbits of points $P\in{\P_1(\bar{K})}$ under $f$, i.e., the sequence of iterates
\[P,f(P),\ldots,f^{\circ{n}}(P),\ldots\,.\] 

Two rational functions $f$ and $g$ in $\bar{K}(x)$ are called $\bar{K}$-\emph{linearly conjugate} if there exists a fractional linear transformation $\ell(x)=(ax+b)/(cx+d) \in{}\pgl_2(\bar{K})$ such that $g = \ell^{-1}\circ{f}\circ{\ell}$. We will use the notation \(f^\ell = \ell^{-1}\circ{f}\circ{\ell}\)
for the action of conjugating $f$ by $\ell$.
Two linearly conjugate rational functions have essentially the same dynamical properties over $\bar{K}$. For instance, a point $P\in{\bar{K}}$ is periodic for $f$ (i.e., there exists an integer $n\geq{1}$ such that $f^{\circ{n}}(P)=P$) if and only if $\ell^{-1}(P)$ is periodic for $f^\ell$. We will denote by $[f]$ the conjugacy class of $f$ under elements of $\pgl_2(\bar{K})$, and by $[f]_K$ the conjugacy class of $f$ under elements of $\pgl_2(K)$.

It may be the case that the field of definition of a function $f$ changes under conjugation; for example, the rational function $f=\sqrt{2}x^2$ is conjugate to $g=2x^2$ under $\ell=\sqrt{2}x$. To capture the idea of a minimal field of definition up to conjugation, Silverman in \cite{silverman1995} defines the \emph{field of moduli}  of $f\in{\bar{K}}(x)$ (over $K$) as the fixed field in $\bar{K}$ of the group
\[G_f = \{\sigma\in{\gal(\bar{K}/K)} : [f^\sigma]=[f]\}.\]
Clearly, the field of moduli is contained in the field of definition of any function in $[f]$.
Silverman proved that there exists a function in $[f]$ whose field of definition is equal to the field of moduli of $f$ if $\deg(f)$ is even, or $f$ is a polynomial map. More recently, the problem of comparing field of moduli vs.\ field of definition has been studied in higher dimension by Doyle and Silverman \cite{doyle-silverman2019}, and by Hutz and Manes \cite{hutz-manes2013}.

In this article we consider the fields of definition of the iterates of $f\in{\bar{K}(x)}$. It may be that the field of definition of some iterate $f^{\circ{n}}$ is strictly contained in the field of definition of $f$. For example, we can take $f(x)=\sqrt[3]{2}x^2$, and then $f^{\circ{2}}(x)=2x^4$. We thus ask what is the ``minimal'' field of definition in the sequence of iterates of $f$. To be precise, we define the \emph{field of iterates of $f$} (over $K$) to be $\bigcap_{i=1}^{\infty}K(f^{\circ i})$. This intersection must stabilize after a finite number of steps since $K\subseteq K(f^{\circ n})\subseteq{K(f)}$ for any integer $n\ge{1}$ and  $K(f)$ is a finite extension of $K$. 
\begin{rem}
The field of iterates of $f$ over $K$ is always a field of definition of an iterate $f^{\circ{N}}$ for some integer $N\ge{1}$. This is a consequence of the fact that $K(f^{\circ{mn}})$ is contained in the intersection $K(f^{\circ{m}})\cap K(f^{\circ{n}})$ and therefore, if the intersection $\bigcap_{i=1}^{\infty}K(f^{\circ i})$ stabilizes at the index $N$ we have that $$\bigcap_{i=1}^{\infty}K(f^{\circ i})\subseteq K(f^{\circ N!}) \subseteq \bigcap_{i=1}^{N}K(f^{\circ i}) = \bigcap_{i=1}^{\infty}K(f^{\circ i}).$$
\end{rem}

We now formulate the main problem we study in this article:

Assuming that for $f\in\bar{K}(x)\setminus K(x)$ and some integer $n\ge{2}$ we have $f^{\circ{n}}\in K(x)$,
 what can be said about $f(x)$? We have already seen above that this problem is not vacuous, i.e., there exist $f\in{\bar{K}(x)}\setminus{K(x)}$ such that $f^{\circ{n}}\in{K(x)}$ for some $n\ge{2}$. 

For any number field $F$ and an integer $n\ge{1}$, we define
\[A_n(F):=\{f(x)\in{\bar{K}(x)}\mid f^{\circ n}\in{F(x)}\},\]
and 
\[A(F):=\bigcup_{n=1}^{\infty}A_n(F).\]
If $F$ is a finite extension of $K$, then the set $A(F)$ is the set of all functions in ${\bar{K}(x)}$ whose field of iterates over $K$ is contained in $F$.

For ease of notation in the statement of our results, we will write 
\[S_n(k)=\sum_{i=0}^{n-1}{k^i},\]
where $k$ and $n$ are positive integers. This quantity is of course equal to $\frac{k^n-1}{k-1}$ when $k\ne{1}$.

We now state the characterization of polynomials that are in the set $A(K)$.

\begin{theorem}\label{thm:main}
Fix integers $d,n\ge{1}$, and
let $f(x)\in{\bar{K}[x]} \setminus K[x]$ be a polynomial of degree $d$. Then $f\in A_n(K)$ if and only if $f$ is $K$-linearly conjugate (by an affine transformation $\ell(x)=\alpha{x}+\beta$) to a polynomial $\tilde{f}$ of the form
\begin{equation}\label{eqn:main}
\tilde{f}(x)=ax^{k}g(x^{t}),
\end{equation}
where $g(x) \in {K[x]}$, $k$ and $t$ are positive integers such that $t|S_n(k)$ and $a^t\in{K}$. Moreover, the field of definition of $f$ is equal to $K(a)$ and it is a simple radical extension of $K$ of degree at most $S_n(d)$; if $\tilde{f}$ is not a monomial then $[K(f):K]\le{d}$.
\end{theorem}

If we don't specify which iterate is defined over the base field, the result can be given a simpler form.

\begin{theorem} \label{thm:mainB}
A non-constant polynomial $f\in{\bar{K}[x]}\setminus K[X]$ is in  ${A(K)}$ if and only if  $f$ is $K$-linearly conjugate to a polynomial $\tilde{f}$ of the form 
\begin{equation}
\tilde{f}(x)=ax^{k}g(x^{t}),
\end{equation}
where $k$ and $t$ are positive coprime integers, $g\in{K[x]}$ and $a^t\in{K}$.
\end{theorem}

A similar characterization does not hold for general rational functions. We denote by $B(K)$ the set of rational functions $f\in{\bar{K}(x)}$ such that $f$ is $K$-linearly conjugate to a rational function $\tilde{f}(x)=ax^{k}g(x^{t})\in{\bar{K}(x)}$ with $k,t$ coprime, $g(x)\in K(x)$ and $a^t\in K$.
We will see that the following Proposition holds.
\begin{prop}
Not all rational functions in $A(K)$ are $K$-linearly conjugate to functions of the shape $ax^{k}g(x^{t})$ as above, i.e.
$$B(K)\subsetneq A(K).$$    
\end{prop}

That $B(K)\subseteq A(K)$ follows from the relevant direction of the proof of Theorem~\ref{thm:mainB}. To see that they are distinct we will consider the function
\begin{equation} \label{eq:counterexample1}
f(x)=\frac{cx^2-2x-c}{x^2+2cx-1},
\end{equation}
where $c = 2-\sqrt{3}$. We will see in Section~\ref{sec:examples} that $f\in A_2(\mathbb{Q})$ but $f\not\in B(\mathbb{Q})$.

Given a rational function $f\in{A(k)}$, there are several ways in which
we can construct other rational functions in $A(k)$ from $f$. We now provide a list of (elementary) ways in which this can be done (we do not claim that this list is exhaustive). 

\begin{prop} \label{prop:conj}
Let $f\in{A_n(K)}$. 
\begin{enumerate}
    \item \label{prop:conj-item1} Then $f^{\circ{k}}\in A_m(K)$, where $m=\frac{n}{\gcd(k,n)}.$
    \item \label{prop:conj-item2} Let $\ell\in{\pgl_2(K)}$ be a fractional linear transformation. Then $f^\ell \in A_n(K)$.
    \item \label{prop:conj-item3} Let $g\in A_m(K)$ such that $f\circ g=g\circ f$. Then $f\circ g\in A_{\lcm(n,m)}(K)$. In particular, if $\ell\in{\pgl_2(\bar{K})}$ is an automorphism of $f$ of finite order $m$, then $\ell\circ f\in A_{\lcm(n,m)}(K)$.
    \item If $\ell\in \pgl_2(\bar{K})$ and $\ell\circ f=f\circ\ell^{\circ h}$ for an integer $h\ge{1}$ and $\ell\in A_{S_{n+1}(h)-1}(K)$ then $\ell\circ f\in A_n(K)$.

    \item \label{prop:conj-item5} Suppose $g\in{K(x)}$ and there exists $\ell\in{\pgl_2(\bar{K})}$ such that $g=\ell^{-1}\circ{f^{\circ{n}}}\circ{\ell}$, i.e. $g$ is a twist of $f^{\circ{n}}$. Then $f^\ell\in{A_n(K)}$.
\end{enumerate}
\end{prop}

Property \ref{prop:conj}~\eqref{prop:conj-item2} is a special case of Property  \ref{prop:conj}~\eqref{prop:conj-item5}, and is false for a general element in $\pgl_2(\bar{K})$. The topic of commuting rational functions, as in Property~\ref{prop:conj}~\eqref{prop:conj-item3}, has been studied extensively by Julia~\cite{julia}, Fatou~\cite{fatou} and Ritt~\cite{ritt}, and more recently by Er\"{e}menko~\cite{eremenko1990} and Pakovich~\cite{pakovich2021}. To summarize, two rational functions $f,g\in{\bar{K}}(x)$ of degree $\ge 2$ commute, i.e. $f\circ{g}=g\circ{f},$ if and only if $f$ and $g$ share a common iterate (i.e., $f^{\circ{m}}=g^{\circ{n}}$ for some positive integers $m$ and $n$) or $f$ and $g$ are $\bar{K}$-linearly conjugate to power maps, Latt\`es maps or  Chebyshev polynomials.

\bigskip

The paper is organized as follows: in Section~\ref{sec:proof.main} we give the proof of Theorems~\ref{thm:main} and \ref{thm:mainB}.
In Section~\ref{sec:examples} we give examples of rational functions in $A(K)\setminus B(K)$ i.e. not of the shape \eqref{eqn:main}. We construct these examples, which are related to Chebyshev polynomials and Lattés maps, using the group law on the circle and on an elliptic curve. 
In Section~\ref{sec:PGL} we focus on rational functions of degree one, and show that all rational functions of degree one in $A(K)$ can be described in a simple way in terms of the eingenvalues of the matrix of their coefficients.

\section*{Acknowledgements}
We thank Pietro Corvaja and Fedor Pakovich for useful suggestions and remarks.
We thank Umberto Zannier and the Scuola Normale di Pisa for support during the preparation of this article.
The first author is a member of the GNSAGA group of the Istituto Nazionale di Alta Matematica.

\section{Proof of the main theorems}\label{sec:proof.main}

We begin with a few lemmas that will be used in the proof of Theorem~\ref{thm:main}. 

\begin{lemma}[M{\"u}ller and Zieve \cite{zieve2008ritt}, Corollary 2.9]\label{lem:zieve-muller1}
Suppose $a,b,c,d\in{\mathbb{C}[x]\setminus{\mathbb{C}}}$, such that $a\circ{b}=c\circ{d}$, where $\deg{a}=\deg{c}$. Then there exists a linear polynomial $\ell\in\mathbb{C}[x]$ such that $a=c\circ{l}$ and $b=\ell^{-1}\circ{d}$. 
\end{lemma}

This immediately generalizes by induction to the following lemma.

\begin{lemma} \label{lem:zieve-muller2}
Let $n$ be a positive integer, and let $f_1,\ldots,f_n,g_1,\ldots,g_n\in \mathbb{C}[x]\setminus\mathbb{C}$. 
If $f_1\circ\dotsb\circ{f_n}=g_1\circ\dotsb\circ{g_n}$ and $\deg(f_i)=\deg(g_i)$ for $1\le{i}\le{n}$ then there exist linear polynomials $\ell_0,\ldots,\ell_n\in\mathbb{C}[x]$ such that $g_i=\ell_{i-1}\circ{f_i}\circ{\ell_i}^{-1}$, for $1\le{i}\le{n}$ where $\ell_0(x)=\ell_n(x)=x$.
\end{lemma}

\begin{lemma} \label{lem:quot-of-two-leading}
 Let $f(x)=\sum_{i=0}^d{a_ix^i}\quad (a_d\neq 0)$ be a polynomial in ${\bar{K}[x]} \setminus K[x]$ of degree $d\ge 1$ such that $f^{\circ n}\in K[x]$.
 Then $f$ is $K$-linearly conjugate to a polynomial $f_1=\sum_{i=0}^d{b_ix^i}$ with $b_{d-1}=0$ and $f_1^{\circ n}\in K[x]$.
\end{lemma}
\begin{proof}

We begin with the case $d=1$. We denote $f(x)=\alpha x+\beta$. First we note that $\alpha\ne 1$; indeed, if $\alpha=1$ then $f^{\circ n}(x)=x+n\beta$, implying $\beta\in{K}$, and thus $f\in{K[x]}$, contradiction. Now, 
\begin{align*}
    f^{\circ n}(x) &= \alpha ^nx+(1+\alpha+\alpha^2+\ldots+\alpha^{n-1})\beta \\
    &= \alpha^nx+\frac{\alpha^n-1}{\alpha-1}\beta.
\end{align*}
From this we deduce that $\alpha^n, \frac{\beta}{\alpha-1}\in K$. Define $\ell(x) = x-\frac{\beta}{\alpha-1} \in K[x]$. Then
$f^\ell(x)=\alpha x$, and $(f^\ell)^{\circ n} = \alpha^nx\in K[x]$ as required.

Now assume $d\ge{2}$. We denote
\[f=\alpha_1x^d+\beta_1x^{d-1}+...,\]
and the $n$-th iterate of $f$ by
\[f^{\circ{n}}=\alpha_nx^{d^n}+\beta_nx^{d^n-1}+....\]
It is now easy to see that we have the following two recursive conditions on $\alpha_n,\beta_n$ for $n\ge{2}$:
\[
\begin{cases}
\alpha_n = \alpha_1(\alpha_{n-1})^d, \\
\beta_n = \alpha_1d\alpha_{n-1}^{d-1}\beta_{n-1}.
\end{cases}
\]
From the first condition we obtain $\alpha_n = \alpha_1^{S_n(d)}$, and from this we get:
\begin{align*}
\beta_n &= \alpha_1d\alpha_1^{(d-1)S_{n-1}(d)}\beta_{n-1} \\
&= d\alpha_1^{d^{n-1}}\beta_{n-1} \\
&= d^{n-1}\alpha_1^{S_n(d)-1}\beta_1 \\
&= d^{n-1}\alpha_1^{S_n(d)-1}\beta_1.
\end{align*}
So, now if $\alpha_n,\beta_n\in{K}$ then so is $\frac{\beta_n}{\alpha_n}=d^{n-1}\frac{\beta_1}{\alpha_1}$, which implies $\frac{\beta_1}{\alpha_1}=\frac{a_{d-1}}{a_d} \in {K}$.
If we set $\ell(x)=x-\frac{a_{d-1}}{d a_d}$ and $f_1=\ell^{-1}\circ f\circ \ell$ one can easily check that $f_1$ is missing the term in $x^{d-1}$ and $f_1^{\circ n}\in K[x]$ by Proposition~\ref{prop:conj}\eqref{prop:conj-item2}.

\end{proof}

\begin{lemma}\label{lem:main-lemma2}
 Let $f=\sum_{i=0}^d{a_ix^i}\quad (a_d\neq 0),g=\sum_{i=0}^d{b_ix^i}$ be polynomials in ${\bar{K}[x]}$ with $d\ge 1$, with $a_{d-1}=b_{d-1}=0$. Suppose that $g=f\circ{\ell}$, where $\ell(x)\in\bar{K}[x]$ is linear, i.e., $\ell(x)=\alpha{x}+\beta$ with $\alpha\ne{0}$, then $\beta=0$.
\end{lemma}

\begin{proof}
We have 
\[g = f\circ{\ell} = a_d(\alpha{x}+\beta)^d+a_{d-2}(\alpha{x}+\beta)^{d-2}+\ldots+a_0.\]
We see that only the first summand contributes to the term $x^{d-1}$. Therefore
\[0 = b_{d-1} = d{a_d}\alpha^{d-1}\beta,\]
so that $\beta$ must be $0$.
\end{proof}

\begin{lemma} \label{lem:main-lemma3}
 Let $f=\sum_{i=0}^d{a_ix^i}\quad (a_d\neq 0),g=\sum_{i=0}^d{b_ix^i}$ be polynomials in ${\bar{K}[x]}$ with $d\ge 1$, with $f^{\circ{n}}=g^{\circ{n}}$. Suppose that $g=\ell\circ f$, where $\ell(x)=\alpha{x}$ with $\alpha\in\bar{K}$, then $\alpha^{S_n(d)} = 1$ (so $\alpha$ is a $S_n(d)$-th root of unity).
\end{lemma}

\begin{proof}
Suppose $d\ge{2}$. By the proof of Lemma~\ref{lem:quot-of-two-leading}, the leading coefficients of $f^{\circ{n}}$ and $g^{\circ{n}}$ are $a_d^{S_n(d)}$ and $b_d^{S_n(d)}$, respectively, so that by the assumption $f^{\circ{n}}=g^{\circ{n}}$ we have $a_d^{S_n(d)}=b_d^{S_n(d)}$. But by the assumption $g=\ell\circ f$, we know that $b_d=\alpha a_d$, so that we get $\alpha^{S_n(d)}=1$. Thus $\alpha$ must be a $S_n(d)$-th root of unity.
\end{proof}

\begin{lemma} \label{lem:nth-iterate}
Let $f(x)=ax^kg(x^t)$ with $g\in{K[x]}$ (resp.\  $g\in{K(x)}$), with $a^t\in{K}$ for some positive integers $k$ and $t$. Let $n\ge{1}$, then
\[f^{\circ n}(x)=a^{S_n(k)}x^{k^n}g_n(x^t),\]
with $g_n\in{K[x]}$ (resp.\  $g_n\in{K(x)}$).
\end{lemma}

\begin{proof}
By induction on $n$. The case $n=1$ is trivial (define $g_1=g$). 

\begin{align*}
f^{\circ{n+1}}(x) &= f(f^{\circ{n}})(x) \\
&= f(a^{S_n(k)}x^{k^n}g_n(x^t)) \\
&= a^{S_{n+1}(k)}x^{k^{n+1}}g_n(x^t)^kg(a^{S_n(k)t}x^{t k^n}g_n(x^{t})^t)
\end{align*}
(in the last equality we used the identity $S_{n+1}(k)=1+kS_n(k)$).
Now define 
\[g_{n+1}(x)=  g_n(x)^kg(a^{S_n(k)t}x^{k^n}g_n(x)^{t})\]
which is in $K[x]$ (resp.\  ${K(x)}$) because $a^{t}\in{K}$ and $g,g_n\in{K[x]}$ (resp.\  $g,g_n\in{K(x)}$) by the inductive hypothesis.
\end{proof}

\begin{proof}[Proof of Theorem~\ref{thm:main}]
Let $f(x)=\sum_{i=0}^d{a_ix^i}\in \bar{K}[x] \setminus K[x]$ be a polynomial of degree $d\ge 1$.

In one direction, suppose that for $\ell(x)=\alpha{x}+\beta \in K[x]$ we have
$f^\ell = \tilde{f}(x)=ax^{k}g(x^{t})$,
where $g(x) \in {K[x]}$, $k$ and $t$ are positive integers such that $t|S_n(k)$ and $a^t\in{K}$. 
By Lemma~\ref{lem:nth-iterate}, we get $(\tilde{f})^{\circ n}\in K[x]$. Now $f^{\circ n} = (\tilde{f}^{\circ n})^{\ell^{-1}} \in K[x]$, so that $f\in{A_n(K)}$.

In the opposite direction, assume $f\in A_n(K)$. By Lemma~\ref{lem:quot-of-two-leading} (up to conjugation by an element of $\pgl_2(K)$) we can assume $a_{d-1}=0$.
Let $G=\gal(\bar{K}/K)$. Since $K(f)\ne{K}$ (we assumed $f\notin K[x]$) there exists $\sigma\in{G}$ such that $f^\sigma\ne{f}$. However, since $f^{\circ{n}}\in{K[x]}$ we have
\[f\circ{\dotsb}\circ{f} = f^\sigma\circ\dotsb\circ f^\sigma.\]
By Lemma~\ref{lem:zieve-muller2} there exist linear polynomials $\ell_1(x)=\alpha_1{x}+\beta_1$ and $\ell_2(x)=\alpha_2{x}+\beta_2$ with $\alpha_1,\alpha_2,\beta_1,\beta_2\in\bar{K}$ such that
\begin{align}
\label{eq:right-ell} f&=f^\sigma\circ{\ell_1}, \\ 
\label{eq:left-ell} f&=\ell_2\circ{f^\sigma}.
\end{align}
Applying Lemma~\ref{lem:main-lemma2} to \eqref{eq:right-ell} we obtain $\beta_1=0$, which implies, by comparing constant coefficients in \eqref{eq:right-ell}, that $a_0^\sigma=a_0$. This, together with \eqref{eq:left-ell} implies that $\beta_2=0$. By Lemma~\ref{lem:main-lemma3}, we get $\alpha_1$ is a root of unity; let $s$ be the exact order of $\alpha_1$.

By combining \eqref{eq:right-ell} and \eqref{eq:left-ell} we get 
\[f=\ell_2\circ{f}\circ{\ell_1^{-1}}.\]
Comparing coefficients in degree $0\le{i}\le{d}$ we get 
\[a_i=\alpha_2 a_i \alpha_1^{-i}.\]
This implies that $\alpha_2$ is an $s$th root of unity, and that the indices $i$ such that $a_i\ne{0}$ belong to an arithmetic progression modulo $s$. If we change the Galois embedding $\sigma$ we will get a different value for $s$, therefore we can write 
\[f(x)=x^kg(x^t),\]
where $t$ is the least common multiple of the values for $s$ as $\sigma$ ranges among all Galois embeddings of $K(f)$ into $\bar{K}$,  $0\le{k}<t$ and $g\in{K(f)[x]}$. 

Now write $g(x)=\sum_{j=0}^m{b_jx^{j}}$, where $b_0,\ldots,b_m\in{K(f)}$, and $b_m \ne 0$.
By \eqref{eq:left-ell} we get
\[{b_j}=\alpha_2{\sigma(b_j)}, \quad j=0,...,m.\]
When $b_j\ne{0}$ this implies
\[\frac{b_j}{\sigma(b_j)}=\alpha_2, \quad j=0,...,m.\]
Therefore for any $i,j\in{0,...,m}$ with $b_ib_j\ne{0}$ we have 
\[\frac{\sigma(b_j)}{b_j}=\frac{\sigma(b_i)}{b_i},\]
or
\begin{equation} \label{eq:ratios}
\frac{b_i}{b_j}=\frac{\sigma(b_i)}{\sigma(b_j)}=\sigma\left(\frac{b_i}{b_j}\right).
\end{equation}
This holds for every $\sigma$ and shows that for any $i,j\in{0,...,m}$ with $b_ib_j\ne{0}$ we have $\frac{b_i}{b_j}\in{K}$. Therefore,
\[g(x)=b_m\tilde{g}(x),\]
where $\tilde{g}(x)\in{K[x]}$, and so
\[f(x)=b_mx^{k}\tilde{g}(x^t),\]
and $K(f)=K(b_m)$. 

For a fixed $\sigma$, we have $\sigma(b_1^s)=(\alpha_2^{-1} b_m)^s=b_m^s,$ therefore $b_m^t\in{K}$ so that $K(f)$ is a simple radical extension of $K$. According to Lemma~\ref{lem:nth-iterate}, the leading coefficient of the $n$th iterate of $f$ is $b_m^{S_n(k)}$; for it to be in $K$ implies that $t$ divides $S_n(k)$. 

As for the extension degree $[K(f):K]$, of course if $\tilde{f}$ is not monomial we have $[K(f):K] = [K(b_1):K] \le t \le d$. Otherwise, in the worst case we have $[K(b_1):K] \le S_n(k) \le S_n(d)$.
\end{proof}

\begin{proof}[Proof of Theorem~\ref{thm:mainB}]
First, we show that $t$ and $k$ are coprime if and only if there exists a positive integer $n$ such that $t|S_n(k)$. If $t|S_n(k)$ for some positive integer $n$ then
\[S_n(k)=1+\ldots+k^{n-1} \equiv 0 \pmod{t}\]
It is clear that $n>1$ since otherwise we get $S_1(k)=1\equiv 0 \pmod{t}$, contradiction. Isolating $1$, we get
\[1 \equiv k(-1-\ldots-k^{n-2}) \pmod{t}.\]
This implies that $k$ is invertible mod $t$, so that $k$ and $t$ are coprime.
In the other direction, suppose $k$ and $t$ are coprime. If $k\equiv 1 \pmod{t}$ then clearly $S_t(k)\equiv 0 \pmod{t}$ so that $t|S_t(k)$.
If $k\not\equiv 1 \pmod{t}$ then $t$ divides $\frac{k^n-1}{k-1}=S_n (k)$ for $n$ equal to the multiplicative order of $k$ modulo $t(k-1)$.

Now assume that $f\in A(K)$, and let $n$ be the least positive integer such that $f\in A_n(K)$. By Theorem \ref{thm:main} $f$ is $K$-linearly conjugate to $ax^{k}g(x^{t})$ with
$g(x) \in {K[x]}$, $t|S_n(K)$ and $a^t\in K$; as shown, $t$ and $k$ are coprime, and we are done.

For the converse, assume that $f$ is conjugate to $ax^{k}g(x^{t})$ with $g(x) \in {K[x]}$, $k$ and $t$ coprime and $a^t\in K$. Then there is a positive integer $n$ such that $t|S_n(k)$, and therefore by Theorem \ref{thm:main} $f\in A_n(K)\subseteq A(K)$.
\end{proof}

\section{Examples of rational functions} \label{sec:examples}

In this section we present a family of rational functions in $A_n(\mathbb{Q})\setminus B(\mathbb{Q})$ based on the group law on $S^1$.

We begin with a lemma which is useful to show that some functions are not in $B(\mathbb{Q})$.

\begin{lemma} \label{lem:periodic}
Let $f(x)=ax^{k}g(x^{t})$ where $k$ and $t$ are positive coprime integers, $g\in{K(x)}$, $a\notin K$ and $a^t\in{K}$. Then $f(\{0,\infty\})\subseteq\{0,\infty\}$.
Furthermore, every function in $B(K)$ has a $K$-rational periodic point of period $1$ or two $K$-rational periodic points of period $2$ which form a $2$-cycle. 
\end{lemma}

\begin{proof}
Write $g(x)=p(x)/q(x)$ with $p(x),q(x)\in K[x]$ coprime (in this case the polynomials $p(x^t),q(x^t)$ are coprime as well).

Observe that either $f(\infty)=0$, if $t\deg q > k+ t \deg p$, or  $f(\infty)=\infty$, if $t\deg q < k+ t \deg p$; in particular $t\deg q \neq k+ t \deg p$ because $k$ and $t$ are coprime and $t>1$ because $a\not\in K$.

Similarly, either $f(0)=0$ if $k$ is smaller than $t$ times the multiplicity of $q$ at $x=0$ or $f(0)=\infty$ otherwise.

Therefore, if neither 0 nor $\infty$ is a fixed point of $f$, they are swapped by $f$ and they are two $K$-rational periodic points of order 2.
 
The last conclusion is straightforward since conjugation by an element in $\pgl_2(K)$ induces a bijection of $K$-rational fixed points and periodic points of period $2$.
\end{proof}

We can now prove the claim stated in the introduction.

\begin{prop}\label{prop:counterexample}
Let $f(x)=\frac{cx^2-2x-c}{x^2+2cx-1},$ where $c = 2-\sqrt{3}$. Then $f(x)\in A_2(\mathbb{Q})$, but $f(x)\notin B(\mathbb{Q})$.
\end{prop}

\begin{proof}
The second iterate of $f$ is
\[f^{\circ{2}}(x)=\frac{x^4 - 4x^3 - 6x^2 + 4x + 1}{x^4 + 4x^3 - 6x^2 - 4x + 1},\]
which is defined over $\mathbb{Q}$; this shows that $f\in{A_2(\mathbb{Q})}$.

We note that the fixed points (i.e., the periodic points of period $1$) of $f$ are $-c, \pm{i}$, none of which are defined over $\mathbb{Q}$, and the periodic points of (exact) period $2$ are $c$ and $1$, so that only one of them is $\mathbb{Q}$-rational. 

Assume by way of contradiction that $f$ is in $B(\mathbb{Q})$.
Then by Lemma~\ref{lem:periodic}, $f$ has a $\mathbb{Q}$-rational fixed point or two $\mathbb{Q}$-rational periodic points of period $2$, but this is not the case. We conclude that $f\notin B(\mathbb{Q})$ as required.
\end{proof}

This example can be better understood using the theory of twists of rational functions. A \emph{twist} of $f\in{K}(x)$ is an equivalence class $[g]_K$ for $g\in{K}(x)$ such that $[f]=[g]$ (i.e., $f$ and $g$ are $\bar{K}$-linearly conjugate, but not necessarily $K$-linearly conjugate); for any $f\in{K}(x)$, the class $[f]_K$ is called the \emph{trivial} twist. We abuse this terminology by calling a function $g$ as in the definition a twist of $f$. An \emph{automorphism} of $f\in\bar{K}(x)$ is an element $\ell\in \pgl_2(\bar{K})$ such that $f^\ell = f$. Twists are intimately related to automorphisms; in particular, if the automorphism group of a rational function $f\in{K(x)}$ is trivial, then $f$ has no non-trivial twists (See \cite{silverman1995} and \cite{silverman2007}*{Sections 4.7 to 4.9} for more details). It turns out that for the function $f$ in \eqref{eq:counterexample1}, the second iterate $f^{\circ{2}}$ is a non-trivial twist under conjugation by $\ell(x)=\frac{x+c}{cx-1}$ of 
\[g(x)=-\frac{4 {\left(x^{3} - x\right)}}{x^{4} - 6 x^{2} + 1}
\]
where $g$ itself is the second iterate of 
\[h(x)=-\frac{2x}{x^{2} - 1}\]
and $f=h^\ell$. The last identity shows that the field of moduli of $f$ is $\mathbb{Q}$.
It is clear that the function $h(x)$ is trivially in $A_2(\mathbb{Q})$ because it is defined over $\mathbb{Q}$, so the fact that $f\in A_2(\mathbb{Q})$ is a consequence of Proposition~\ref{prop:conj}\eqref{prop:conj-item5}.

\subsection{A family of functions based on rotations}
We now use the group law on $S^1$ to give an example of a \emph{family} of non-polynomial rational functions $f\in{\bar{\mathbb{Q}}(t)}$ such that $f\in{A_n(\mathbb{Q})}\setminus{B(\mathbb{Q})}$. 

Let $k\ge{1}$ be an integer, and let $\varphi\in\mathbb{R}$. Consider the self map of the unit circle $S^1$ defined by
\[ \Psi: (\cos\theta,\sin\theta)\mapsto (\cos(k\theta+\varphi),\sin(k\theta+\varphi)) \]

We can project this map down to a rational function on $\P_1$ in the following way: let $(x,y)=(\cos\theta,\sin\theta)$, and let $t$ be $\tan(\theta/2)$. Then we get $(x,y)=(\frac{1-t^2}{1+t^2},\frac{2t}{1+t^2})$, and $t=\frac{1-x}{y}$. Thus the map $\Psi$ descends to a map $f:\P_1\to\P_1$
\[
\begin{CD}
S^1 @>\Psi>> S^1\\
@VVV @VVV\\
\P_1 @>f>> \P_1
\end{CD}
\]
We show that the map $\Psi$ can be expressed as a rational function in $x$ and $y$, and thus $f$ as a rational function in $t$.

Recall the definition of the Chebyshev polynomials: the  $k$-th Chebyshev polynomial of the first kind is the polynomial $T_k$ such that $\cos(k\theta)=T_k(\cos\theta)$, while the  $k$-th Chebyshev polynomial of the second kind is the polynomial $U_k$ such that $\sin((k+1)\theta)=U_k(\cos\theta)\sin\theta$.

Then we have:
\[\cos(k\theta+\varphi)=T_k(x)\cdot\cos\varphi - U_{k-1}(x)\cdot{y}\cdot\sin\varphi,\]
and
\[\sin(k\theta+\varphi)=U_{k-1}(x)\cdot{y}\cdot\cos\varphi+T_k(x)\sin\varphi.\]

So that we get
\begin{equation}\label{eqn:rotation} f(t)=\frac{2tU_{k-1}(\frac{1-t^2}{1+t^2})\cdot\sin\varphi-(1+t^2)T_k(\frac{1-t^2}{1+t^2})\cdot\cos\varphi+(1+t^2)}{2tU_{k-1}(\frac{1-t^2}{1+t^2})\cdot\cos\varphi+(1+t^2)T_k(\frac{1-t^2}{1+t^2})\cdot\sin\varphi}.\end{equation}

The rational function $f$ acts on $\P_1$ (identified with $\mathbb{R}\cup \{\infty\}$)  by sending $t=\tan(\theta/2)$ to $f(t)=\tan(\frac{k\theta+\varphi}{2})$. 

We remark that the function \eqref{eq:counterexample1} is of this shape for $k=2,\varphi=\pi/6$.

We now consider only those $\varphi$ such that $\cos\varphi, \sin\varphi$ are algebraic numbers. Then the field of definition of $f$ is $\mathbb{Q}(\cos\varphi,\sin\varphi)$.

\begin{prop} \label{prop:family}
Fix $n,k\geq 2$ and let $\varphi=\frac{\pi}{S_n(k)}$.
Then the function $f$ is in $A_n(\mathbb{Q})\setminus (\mathbb{Q}(t) \cup B(\mathbb{Q}))$.
\end{prop}

\begin{proof}
The field of definition of $f$ is $\mathbb{Q}(\cos\varphi,\sin\varphi)$. 

For $n,k\ge {2}$ we have $0<\varphi\le\pi/3$, so that by Niven's theorem~\cite{book:niven}*{Corollary 3.12} we get that $\mathbb{Q}(\cos\varphi,\sin\varphi)\ne\mathbb{Q}$ and  $f\not\in \mathbb{Q}(t)$. 

If we iterate the function $n$ times we get $f^{\circ n}(t)=\tan(\frac{k^n\theta+S_n(k)\varphi}{2})$, which is again of the shape \eqref{eqn:rotation} with $k^n, S_n(k)\varphi$ in place of $k,\varphi$. 
The field of definition of $f^{\circ n}$ is then $\mathbb{Q}(\cos(S_n(k)\varphi),\sin(S_n(k)\varphi))=\mathbb{Q}$, so $f\in A_n(\mathbb{Q})$.

We can use the same argument as in the proof of Proposition \ref{prop:counterexample} to show that $f\not\in B(\mathbb{Q})$.

The fixed points of $f^{\circ 2}$ (which consist of the periodic points of $f$ of period $1$ and $2$) are the values $\tan\frac{\theta}{2}$ such that $\tan\frac{\theta}{2}=\tan\frac{k^2\theta+(k+1)\varphi}{2}$. This happens if and only if
\begin{equation*}
    \frac{\theta}{2}\equiv \frac{k^2\theta+(k+1)\varphi}{2} \pmod{\pi}
\end{equation*}
 which corresponds to 
\[\theta\equiv -\frac{\varphi}{k-1} \equiv -\frac{\pi}{k^n-1} \pmod{\frac{2\pi}{k^2-1}}.\]
 By Niven's theorem, in order for $t=\tan\frac{\theta}{2}$ to be  $\mathbb{Q}$-rational $\theta$ must be an integral multiple of $\pi/2$.
 
In this case we would have integers $m,h$ such that
\begin{align*}
    -\frac{\pi}{k^n-1}+\frac{2\pi }{k^2-1}m&=\frac{\pi}{2}h.
\end{align*}
Dividing by $\pi$ and multiplying by $2(k+1)(k^n-1)$ we get 
\begin{align*}
    -2(k+1)+4S_n(k)m&=(k^2-1)S_n(k)h.
\end{align*}

This is not possible if $n>2$, because the equation implies that $2(k+1)$ is a nonzero multiple of $S_n(k)$ but in this case $2(k+1)<S_n(k)$. Thus Lemma~\ref{lem:periodic} shows that in this case $f\not\in B(\mathbb{Q})$.

Assume now that $n=2$; then the equation becomes
\begin{align*}
    m&=\frac{(k^2-1)h+2}{4}.
\end{align*}
This has no solution if $k$ is odd, while it has solutions
\[
m\equiv\frac{k^2}{2}\pmod{(k^2-1)}
\]
if $k$ is even.

This corresponds to the only value 
\[
\theta\equiv -\frac{\varphi}{k-1}+\frac{k^2}{2}\frac{2\pi}{k^2-1}\equiv \pi \pmod{2\pi}
\]
and so to the only point $t=\infty$.

Therefore according to Lemma~\ref{lem:periodic}, if $f\in B(\mathbb{Q})$ then $\infty$ must be a fixed point of $f$.
However one can check that
$$f(\infty)=f\left(\tan\frac{\pi}{2}\right)=\tan\left(\frac{k\pi+\varphi}{2}\right)=\tan\left(\frac{\varphi}{2}\right)\neq\infty\qquad \text{ if $k$ is even.}$$

This shows that in this case too $f\not\in B(\mathbb{Q})$.
\end{proof}

\subsection{Elliptic curves}

We can define a similar construction to the one presented in the previous subsection but on an elliptic curve instead of $S^1$. However, as we shall see we will be much more limited when it comes to projecting it to $\P_1$.

Let $E$ be an elliptic curve defined over a number field $K$ by a Weierstrass equation. For a fixed integer $d$ the multiplication by $d$ map, denoted by $[d]$, is an endomorphism of $E$ defined over $K$. As it commutes with the involution $[-1]$, it descends via the $x$ coordinate to a rational function $\Phi_d$ over $\P_1$ (defined over $\mathbb{Q}$). These two maps satisfy the following commutative diagram:
\[
\begin{CD}
E @>[d]>> E\\
@VxVV @VVxV\\
\P_1 @>\Phi_d>> \P_1
\end{CD}
\]
The rational function $\Phi_d$ is an example of a Lattès map (see \cite{article:milnor} for background on Lattès maps). 

Now fix positive integers $n,d\ge 2$ and take $Q \in E(\bar{K})$ an $m$-torsion point on the elliptic curve, where $m$ divides $S_n(d)$. Let us consider the map $\Psi: P\mapsto dP+Q$, which is defined over $\bar{K}$. It is clear that $\Psi^{\circ n}(P)=d^nP+S_n(d)Q=d^nP$ (compare with $f^{\circ n}(t)$ in Proposition~\ref{prop:family}), which is defined over $K$. Note that $\Psi$ commutes with $[-1]$ if and only if $Q$ is a $m$-torsion point, so that $\Psi$ descends to a rational function $f$ over $\P_1$ only for the case $m=2$. The maps $\Psi$ and $f$ satisfy the following commutative diagram:
\[
\begin{CD}
E @>\Psi>> E\\
@VxVV @VVxV\\
\P_1 @>f>> \P_1
\end{CD}
\]
It is now clear from the above that $f\in A_2(K)$.
\begin{prop} \label{prop:elliptic}
Let $E$ be an elliptic curve over $\Q$ defined by the equation $y^2=p(x)$, with $p$ an irreducible polynomial in $\Q[x]$. 
Fix an odd integer $d\geq 3$ and let $Q$ be a 2-torsion point.
Then the function $f$ is in $A_2(\mathbb{Q})\setminus (\mathbb{Q}(t) \cup B(\mathbb{Q}))$.
\end{prop}
\begin{proof}
    Clearly $2\mid S_2(d)=d+1$, so as argued above with $n=2$ we have that $f\in A_2(\Q)$. The function $f$ is not in $\Q(t)$ because $Q$ is not defined over $\Q$. To show that $f\not\in B(\Q)$ we use again Lemma~\ref{lem:periodic}.
    Let $x_0\in\Q$ be a fixed point of the map $f$. Then $x_0$ is the $x$-coordinate of a point $P$ such that $dP+Q=\pm P$ on $E$. This is impossible because $Q=(d \pm 1)P$ is defined on an extension of degree 3 of $\Q$, while $P$ on an extension of degree at most 2. Therefore the map $f$ does not have any rational fixed points.

    Now let $x_0,x_1\in\Q$ be two 2-periodic points for $f$ which form a 2-cycle. As above, $x_0$ is the $x$-coordinate of a point $P$ and $x_1$ the $x$-coordinate of one of the points $\pm d P+Q$. But this is again impossible, because $P$ and $\pm dP+Q$ are defined over extensions of degree at most 2, while $Q=(dP+Q)-[d]P$ has degree 3 over $\Q$.

    By Lemma~\ref{lem:periodic} this shows that $f\not\in B(\Q)$.
\end{proof}

\begin{rem}
    Notice that the function constructed in this way is not one of the functions considered in Proposition~\ref{prop:family}. If it were we would have $n=2,k=d^2$ in Proposition~\ref{prop:family}, and then the field of definition would be $\Q(\sin\frac{\pi}{d^2+1},\cos\frac{\pi}{d^2+1})$, which has degree over $\Q$ greater than 3 for $d\geq 3$.
\end{rem}

We show here an explicit example of this construction. For simplicity we take a curve with one rational 2-torsion point, so that the example can be defined over a field of degree 2 rather than 3.

\begin{exmp}
Let $E$ be the elliptic curve defined by the Weierstrass equation
\[y^2 = x^3-2x,\]
and let $Q_0 = (\sqrt{2},0)\in E(\Q)[2]$ be the chosen $2$-torsion point.
Then the map $\Psi(P) = 3P+Q$ descends to a rational function $f$ of the form 
\tiny
\[
    f(x) = 
    \frac{\sqrt{2} x^9+18 x^8+24 \sqrt{2} x^7-144 x^6+120 \sqrt{2} x^5+240 x^4-288 \sqrt{2} x^3+192 x^2+144 \sqrt{2} x+32}
    {x^9-9 \sqrt{2} x^8+24 x^7+72 \sqrt{2} x^6+120 x^5-120 \sqrt{2} x^4-288 x^3-96 \sqrt{2} x^2+144 x-16 \sqrt{2}}.
\]
\normalsize
Taking the second iterate, we get:
 \[
  f^{\circ 2}(x) = 
 \frac{x^{81}+2160 x^{79}+1104624 x^{77}+\ldots+89060441849856 x}
 {81 x^{80}-37584 x^{78}+2048112 +\ldots+1099511627776}.
 \]
It is possible to compute that $0$ and $\infty$ are the only two $\Q$-rational periodic points of period $\le 2$. They are part of two distinct cycles of length 2: $\{0,-\sqrt{2}\}$ and $\{\infty,\sqrt{2}\}$. Therefore by Lemma~\ref{lem:periodic} we conclude $f\notin B(\Q)$.
\end{exmp}

\section{Fractional linear transformations}\label{sec:PGL}
In this section we study the automorphisms of $\P_1$ which are in $A(K)$. We will identify them with matrices in $\PGL_2(K)$.

As in the previous sections, $K$ will be a fixed number field. For an element $A\in\GL_2(K)$ we denote by $[A]$ its equivalence class in $\PGL_2(K)$.

\begin{exmp}
Fix a diagonalizable matrix $B\in\GL_2(K)$, so that
\[
B=MDM^{-1}                                                                              
\]
for some matrices $M,D\in\GL_2(K')$, with $D$ diagonal and $K'$ some extension of $K$ with $[K':K]\leq 2$.

For any $n>1$ we can find a diagonal matrix $\Delta\in\GL_2(F)$ such that $\Delta^n=D$, where $F$ is an extension of $K'$ of degree at most $n^2$, obtained by extracting $n$-th roots of the two elements on the diagonal of $D$.

We can now define $A=M\Delta M^{-1}\in \GL_2(F)$ and we have that $A^n=B\in\GL_2(K)$. This provides many examples of matrices with coefficients in a number field such that a power is defined over a smaller number field.
\end{exmp}

We will show that these examples exhaust all instances of automorphisms of $\P_1$ defined over a number field, having an iterate defined over a smaller number field.

\begin{rem}
 If a matrix $A$ defined over $F$ has a power defined over a proper subfield then $\det (A)$ has a power in a proper subfield of $F$.
\end{rem}

\begin{prop}\label{prop:matricesI}
Let $A\in\GL_2(\Qbar)$ be a matrix, not diagonalizable over $\Qbar$. Assume that $[A^n]\in\PGL_2(K)$ for some positive integer $n$. Then $[A]\in\PGL_2(K)$.
\end{prop}

\begin{proof}
  Let $\lambda\ne 0$ be the only eigenvalue of $A$. Then there is a matrix $M\in\GL_2(\Qbar)$ such that
 \[
  A=M\begin{pmatrix}\lambda & 1 \\ 0 & \lambda \end{pmatrix} M^{-1}
 \]
Notice that $\begin{pmatrix}\lambda & 1 \\ 0 & \lambda \end{pmatrix}=\begin{pmatrix} 1/\lambda & 0 \\ 0 & 1 \end{pmatrix}\begin{pmatrix}\lambda & \lambda \\ 0 & \lambda \end{pmatrix}\begin{pmatrix}1/\lambda & 0 \\ 0 & 1 \end{pmatrix}^{-1}$; writing $N=M\begin{pmatrix} 1/\lambda & 0 \\ 0 & 1 \end{pmatrix}$ then we have
\[
 [A]=[N]\begin{bmatrix} 1 & 1 \\ 0 & 1 \end{bmatrix}[N]^{-1}
\]
with $N=\begin{pmatrix} a & b \\ c & d \end{pmatrix}\in\GL_2(\Qbar)$.
 
Noticing that $[N]^{-1}=\begin{bmatrix} d & -b \\ -c & a \end{bmatrix}$ and that  $\begin{pmatrix} 1 & 1 \\ 0 & 1 \end{pmatrix}^n=\begin{pmatrix} 1 & n \\ 0 & 1 \end{pmatrix}$ for every integer $n$, we have that
 \[
  [A^n]=\begin{bmatrix} 
         ad-bc -acn & a^2 n \\
         -c^2 n & ad-bc+acn
        \end{bmatrix}.
 \]
One of $a,c$ must be different from 0; assume $c\neq 0$ (the other case being analogous) and $n\neq 0$.
Then $[A^n]$ is defined over $K$ if and only if the following three conditions are satisfied
\begin{equation}\left\{\begin{aligned}\label{matrix.nondiagon.conditions}
 \frac{1}{n}\left(\frac{ad-bc}{c^2}\right)-\frac{a}{c} &\in K\\
 \left(\frac{a}{c}\right)^2 &\in K\\
 \frac{1}{n}\left(\frac{ad-bc}{c^2}\right)+\frac{a}{c} &\in K.
\end{aligned}\right.\end{equation}
Therefore by taking the difference of the first and the third conditions we see that $a/c\in K$ and thus also $(ad-bc)/c^2\in K$.
This implies that conditions \eqref{matrix.nondiagon.conditions} are also satisfied for $n=1$, and thus $[A]$ itself is defined over $K$. 
\end{proof}

\begin{prop}\label{prop:matricesII}
 Let $A\in\GL_2(\Qbar)$ be a matrix diagonalizable in $\Qbar$; let $\lambda_1,\lambda_2$ be its eigenvalues. Let $n$ be a positive integer such that $[A^n]\in\PGL_2(K)$. Then for $K'=K((\lambda_1/\lambda_2)^n)$ we have $[K':K]\leq 2$.
\end{prop}

\begin{proof}
 Write $\alpha=\lambda_1/\lambda_2$. From the hypothesis,
 \begin{align*}
  [A]&=[M]\begin{bmatrix} \alpha & 0 \\ 0 & 1\end{bmatrix} [M]^{-1} \\
  [A^n]&=[M]\begin{bmatrix} \alpha^n & 0 \\ 0 & 1\end{bmatrix} [M]^{-1}
 \end{align*}
for some matrix $M=\begin{pmatrix} a & b \\ c & d \end{pmatrix}\in\GL_2(\Qbar)$. We have $[M]^{-1}=\begin{bmatrix} d & -b \\ -c & a \end{bmatrix}$ so that
\[
 [A^n]=\begin{bmatrix} \alpha^n ad -bc & ab(1-\alpha^n) \\ -cd (1-\alpha^n) & ad -\alpha^n bc \end{bmatrix}.
\]
If $\alpha^n=1$ the statement is true as $K'=K$, so we assume now that $\alpha^n\neq 1$; assume for the moment that $cd\neq 0$.
Then $[A^n]$ is defined over $K$ if and only if the following three conditions are satisfied
\begin{equation} \label{eq:A^n-defined-K}
\left\{
\begin{aligned}
 \frac{a}{c}\frac{\alpha^n}{1-\alpha^n}-\frac{b}{d}\frac{1}{1-\alpha^n}=u_1 &\in K\\
 \frac{ab}{cd} =u_2 &\in K\\
 -\frac{b}{d}\frac{\alpha^n}{1-\alpha^n}+\frac{a}{c}\frac{1}{1-\alpha^n} = u_3 &\in K.
\end{aligned}
\right.
\end{equation}
Taking the difference of the first and the third conditions, one gets that
\[
 \frac{a}{c}+\frac{b}{d}=u_3-u_1 \in K,
\]
which, combined with the second one gives that $a/c , b/d $ lie in an extension $K'$ of $K$ of degree at most 2.
Then, if $\frac{a}{c}+\frac{b}{d}\neq 0$, one can solve the first and the third equations finding
\begin{align*}
\left\{
\begin{aligned}    
 \frac{\alpha^n}{1-\alpha^n} & = \left(u_1 \frac{a}{c} + u_3 \frac{b}{d}\right)/\left(\frac{a^2}{c^2} - \frac{b^2}{d^2}\right) \\
 \frac{1}{1-\alpha^n} & = \left(u_3 \frac{a}{c} + u_1 \frac{b}{d}\right)/\left(\frac{a^2}{c^2} - \frac{b^2}{d^2}\right), 
 \end{aligned}
\right.
\end{align*}
so that
\[
 \alpha^n  = \frac{u_1 \frac{a}{c} + u_3 \frac{b}{d}}{u_3 \frac{a}{c} + u_1 \frac{b}{d}}\in K',
\]
which proves the statement.
On the other hand, if $\frac{a}{c}+\frac{b}{d}= 0$ then solving the first equation in \eqref{eq:A^n-defined-K} for $\alpha^n$ gives us 
\[
 \alpha^n=\frac{\frac{c}{a}u_{1} - 1}{\frac{c}{a}u_{1} + 1} \in K'.
\]
This completes the proof in the case that $cd\neq 0$. If $ab\neq 0$ the argument is analogous. The only cases left are if $a=d=0$ or $b=c=0$, but in these cases $[A^n]$ is equal to $\begin{bmatrix}-bc & 0 \\ 0 & -bc \alpha^n \end{bmatrix}$ or $\begin{bmatrix}ad\alpha^n & 0 \\ 0 & ad  \end{bmatrix}$ respectively, and in either case the statement follows immediately.
\end{proof}

\begin{prop}\label{prop:matricesIII}
Let $A\in \GL_2(\Qbar)$ be a matrix and $n\geq 1$ an integer such that $[A^n]\in\PGL_2(K)$ but $[A]$ is not defined over $K$.
Then there exist extensions $K\subseteq K'\subseteq F$ where $[K':K]\leq 2$ and $[F:K']\leq n$, a diagonal matrix $D\in\GL_2(\Qbar)$ and an invertible matrix $M\in\GL_2(K')$ such that $A=MDM^{-1}$,
 $[D^n]\in \PGL_2(K')$ and 
  $[D] \in \PGL_2(F)$.
\end{prop}

\begin{proof}
By Proposition~\ref{prop:matricesI} the matrix $A$ is diagonalizable over $\Qbar$.
Let $\lambda_1,\lambda_2$ be the two  eigenvalues of $A$; they must be distinct, otherwise $[A]=\left[\begin{pmatrix}1 & 0 \\ 0 & 1\end{pmatrix}\right]$ would be defined over $\mathbb{Q}$. Let $D=\begin{pmatrix}\lambda_1 & 0 \\ 0 & \lambda_2\end{pmatrix}$; there is a matrix $N\in\GL_2(\Qbar)$ such that $A=NDN^{-1}$.

Let $K'=K((\lambda_1/\lambda_2)^n)$ and $F=K'(\lambda_1/\lambda_2)$. Clearly $[F:K']\leq n$, $[D^n]\in \PGL_2(K')$, $[D]\in \PGL_2(F)$ and by Proposition~\ref{prop:matricesII} we know $[K':K]\leq 2$.

Therefore 
\[
[A^n]=[M D^n M^{-1}]
\]
for a matrix $M\in\GL_2(K')$.
Then we have that
\begin{align*}
\left\{
\begin{aligned}
    A^n &= N D^n N^{-1},\\
    A^n &= s M D^n M^{-1},
\end{aligned}
\right.
\end{align*}
for a scalar $s\ne 0$.
The matrices $D^n$ and $sD^n$ are similar, therefore they have the same eigenvalues, so that $s=\pm 1$.
If $s=1$ then we see that $M^{-1}N$ commutes with $D^n$, which is a diagonal matrix with distinct entries; therefore $N=MR$ for some diagonal matrix $R$, and $A=NDN^{-1}=MRDR^{-1}M^{-1}=MDM^{-1}$ as $R$ commutes with $D$.

If $s=-1$ then we must have $\lambda_1^n=-\lambda_2^n$.
Similarly to the case discussed above, $M^{-1}N D^n (M^{-1}N)^{-1}=-D^n$ implies that $N=MR$ for an anti-diagonal matrix $R$. This gives
\[A=NDN^{-1}=MRDR^{-1}M^{-1}=MD'M^{-1}=M\begin{pmatrix}0&1\\1&0\end{pmatrix}D\begin{pmatrix}0&1\\1&0\end{pmatrix}M^{-1},\]
where $D'=\begin{pmatrix}\lambda_2 & 0 \\ 0 & \lambda_1\end{pmatrix}$. Thus we can take $M\begin{pmatrix}0&1\\1&0\end{pmatrix}$ for the $M$ in the statement.
\end{proof}

\bibliographystyle{amsalpha}
\bibliography{foi} 

\end{document}